\documentclass[a4paper, 12pt]{article}
\usepackage{}

\usepackage{mathrsfs}
\usepackage{epsfig}
\usepackage{amsmath}
\usepackage{amssymb}
\usepackage{latexsym}
\usepackage{amsfonts}
\usepackage{amsthm}

\usepackage[numbers,sort&compress]{natbib}
\usepackage{graphicx}
\ifx\pdfoutput\undefined  \else
 \fi

\usepackage[centerlast]{caption2}

\usepackage{color}

\usepackage{txfonts}
\usepackage{amssymb}
\usepackage{mathrsfs}
\usepackage{amsfonts}

\newtheorem{theorem}{Theorem}[section]
\newtheorem{lemma}[theorem]{Lemma}
\newtheorem{cor}[theorem]{Corollary}
\newtheorem{prop}[theorem]{Proposition}

\newtheorem{remark}[theorem]{Remark}

\usepackage{latexsym,amssymb}
\usepackage{amsmath}

\begin{document}

\title{Additively irreducible sequences in commutative semigroups}

\author{
Guoqing Wang\\
\small{Department of Mathematics, Tianjin Polytechnic University, Tianjin, 300387, P. R. China}\\
\small{Email: gqwang1979@aliyun.com}
}

\date{}
\maketitle

\begin{abstract} Let $\mathcal{S}$ be a commutative semigroup, and let $T$ be a sequence of
terms from the semigroup $\mathcal{S}$. We call $T$ an (additively) {\sl irreducible} sequence provided that no sum of its some terms vanishes.
Given any element $a$ of $\mathcal{S}$, let ${\rm D}_a(\mathcal{S})$ be the largest length of the irreducible sequence such that the sum of all terms from the sequence is equal to $a$.
In case that any ascending chain of principal ideals starting from the ideal $(a)$ terminates in $\mathcal{S}$, we found the sufficient and necessary conditions of ${\rm D}_a(\mathcal{S})$ being finite, and in particular, we gave sharp lower and upper bounds of ${\rm D}_a(\mathcal{S})$ in case ${\rm D}_a(\mathcal{S})$ is finite. We also applied the result to commutative unitary rings. As a special case, the value of ${\rm D}_a(\mathcal{S})$ was determined when $\mathcal{S}$ is the multiplicative semigroup of any finite commutative principal ideal unitary ring.

\end{abstract}


\section {Introduction}

We begin this section with some notations in Factorization Theory, which were introduced by A. Geroldinger and F. Halter-Koch \cite{GH} and now have been also widely used in the research of additive problems associated with sequences in groups.

Throughout this paper, we always denote $\mathcal{S}$ to be a commutative semigroup. For any commutative ring $R$, we denote $\mathcal{S}_R$ to be the multiplicative semigroup of the ring $R$. The operation of the semigroup $\mathcal{S}$ is denoted by ``$+$".
The identity element of $\mathcal{S}$, denoted $0_{\mathcal{S}}$ (if exists), is the unique element $e$ of
$\mathcal{S}$ such that $e+a=a$ for every $a\in \mathcal{S}$. The zero element of $\mathcal{S}$, denoted
$\infty_{\mathcal{S}}$ (if exists), is the unique element $z$ of $\mathcal{S}$ such that $z+a=z$ for every
$a\in \mathcal{S}$. If $\mathcal{S}$ has an identity element $0_{\mathcal{S}}$, we call $\mathcal{S}$ a monoid and let
$${\rm U}(\mathcal{S})=\{a\in \mathcal{S}: a+a'=0_{\mathcal{S}} \mbox{ for some }a'\in \mathcal{S}\}$$ be the group of units
of $\mathcal{S}$.
Let
$$\begin{array}{llll}\mathcal{S}^{0}=\left\{\begin{array}{llll}
               \mathcal{S},  & \mbox{if \ \ } \mathcal{S} \mbox{ has an identity element};\\
               \mathcal{S}\cup \{0\},  & \mbox{if \ \ } \mathcal{S} \mbox{ does not have an identity element,}\\
              \end{array}
              \right.
\end{array}$$
be the monoid by adjoining an identity element to $\mathcal{S}$ only when necessary.
Let $$\mathcal{S}^{\bullet}=\mathcal{S}^{0}\setminus \{0_{\mathcal{S}}\}.$$
Let ${\cal F}(\mathcal{S})$
be the free commutative monoid, multiplicatively written, with basis
$\mathcal{S}$. Then any $T\in {\cal F}(\mathcal{S})$, say $T=a_1a_2\cdot\ldots\cdot a_{\ell}$, is a sequence of all
terms $a_i$ from $\mathcal{S}$, which can be also denoted as $$T=\prod\limits_{a\in \mathcal{S}}a^{[{\rm v}_a(T)]},$$ where $[{\rm v}_a(T)]$ means that the element $a$ occurs ${\rm v}_a(T)$ times in the sequence $T$. By $|T|$ we denote the length of the sequence, i.e., $$|T|=\sum\limits_{a\in \mathcal{S}}{\rm v}_a(T)=\ell.$$
By $\varepsilon$ we denote the
{\sl empty sequence} in $\mathcal{S}$ with $|\varepsilon|=0$.
By $\cdot$ we always denote the concatenation operation of sequences. Let $T_1,T_2\in \mathcal{F}(\mathcal{S})$ be two sequences. We call $T_2$
a subsequence of $T_1$ if $${\rm v}_a(T_2)\leq {\rm v}_a(T_1)\ \ \mbox{for each element}\ \ a\in \mathcal{S},$$ denoted by $$T_2\mid T_1,$$ moreover, we write $$T_3=T_1  T_2^{[-1]}$$ to mean the unique subsequence of $T_1$ with $T_2\cdot T_3=T_1$.  We call $T_2$ a {\sl proper} subsequence of $T_1$ provide that $T_2\mid T_1$ and $T_2\neq T_1$. In particular, the  empty sequence  $\varepsilon$ is the proper subsequence of every nonempty sequence. Let $$\sigma(T)=a_1+\cdots +a_{\ell}$$ be the sum of all terms from $T$.
We call $T$ a {\bf zero-sum} sequence, provide that $\mathcal{S}$ is a monoid and $\sigma(T)=0_{\mathcal{S}}$.
In particular,
if $\mathcal{S}$ is a monoid,  we allow $T=\varepsilon$ to be empty and adopt the convention
that $$\sigma(\varepsilon)=0_\mathcal{S}.$$ If the sequence $T$ contains no nonempty zero-sum subsequence, we call $T$ {\bf zero-sum free}. The sequence $T$ is called an {\bf additively reducible} (reducible) sequence if $T$ contains a proper subsequence $T'$ with $\sigma(T')=\sigma(T)$, and is called an {\bf additively irreducible} (irreducible) sequence if otherwise.

The additive properties of sequences in abelian groups (mainly in finite abelian groups) have been widely studied, since  H. Davenport \cite{Davenport} in 1966 and  K. Rogers \cite{rog1} in 1963 independently proposed one combinatorial invariant, denoted ${\rm D}(G)$, for any finite abelian group $G$, which is defined as the smallest $\ell\in \mathbb{N}$ such that
every sequence $T\in \mathcal{F}(G)$ of length $|T|$ at least $\ell$ contains a nonempty zero-sum subsequence.
Although Davenport proposed this invariant to study the algebraic number theory since he observed this invariant ${\rm D}(G)$ is the maximal number of prime ideals which can appear in
the factorization of an irreducible number in a number field, which class group is $G$, the researches on Davenport constant have influenced other fields in Number Theory and in Combinatorics. For example, the Davenport constant has been applied by Alford,  Granville and Pomeranceto \cite{Cnumber} to prove that there are infinitely many Carmichael numbers and by Alon \cite{AlonJctb} to prove the existence of regular subgraphs. What is more important, a lot of researches were stimulated by the Davenport constant together with another famous theorem obtained by P. Erd\H{o}s, A. Ginzburg and A. Ziv \cite{EGZ} in 1961 on additive properties of sequences in groups, which have been developed into a branch, called Zero-sum Theory (see \cite{GaoGeroldingersurvey} for a survey), in Additive Group Theory. In the past five decades, many researchers made efforts to find the values of Davenport constant for finite abelian groups. Unfortunately, the precise values of this constant was known for only a small number of families of finite abelian groups  by far (see \cite{GLP12} for the recent progress).

Note that in any finite abelian group $G$, the sequence $T\in \mathcal{F}(G)$ contains a nonempty zero-sum subsequence if and only if $T$ is reducible. Hence, we have
\begin{equation}\label{equation D(G)=max(T)}
{\rm D}(G)=\max\limits_{T}\{\ |T|\ \}+1,
\end{equation}
 where $T$ takes over all irreducible sequences in the finite abelian group $G$. Accordingly, M. Ska{\l}ba  \cite{SkalbaGraz,Skalba,SkalbaEuropean} formulated an invariant associated with irreducible sequences in any finite abelian group $G$. For any element $g\in G^{\bullet}$,  let ${\rm D}_g(G)$ be the largest length of irreducible sequences $T$ with $\sigma(T)=g$, which is called the relative Davenport constant of $G$ with respect to the element $g\in G^{\bullet}$.
 Hence, \eqref{equation D(G)=max(T)} is equivalent to
 \begin{equation}\label{equation D(G)=max(Dg(G))}
{\rm D}(G)=\max\limits_{g\in G^{\bullet}}\{{\rm D}_g(G)\}+1.
\end{equation}
Ska{\l}ba \cite{Skalba, SkalbaEuropean} determined the precise values ${\rm D}_g(G)$ for any $g\in G^{\bullet}$ in case that $G$ is a finite cyclic group or a finite abelian group of rank two. More importantly, he found the following general bounds:

\medskip

\noindent \textbf{Theorem A.} (\cite{Skalba}) \ {\sl If $G$ is a finite abelian group and $g\in G^{\bullet}$, then
$$\frac{1}{2}{\rm D}(G)\leq {\rm D}_g(G)\leq {\rm D}(G)-1.$$ }

With respect to the classical Davenport constant and the relative Davenport constant defined by Ska{\l}ba, the author of this manuscript
together with W.D. Gao,
formulated the definitions of the Davenport constant and the relative Davenport constant for commutative semigroups, and made some closed related researches on additive properties of sequences in semigroups (see \cite{AdhikariGaoWang14,wangDavenportII,wang,wanggao}).

\medskip

\noindent \textbf{Definition B.}  (see \cite{wangDavenportII,wanggao}) \ {\sl  Define the Davenport constant of the commutative semigroup $\mathcal{S}$, denoted ${\rm D}(\mathcal{S})$, to be the smallest $\ell\in \mathbb{N}\cup\{\infty\}$ such that every sequence $T\in \mathcal{F}(\mathcal{S})$ of length at least $\ell$ is reducible. For any element $a\in \mathcal{S}^{\bullet}$, we define the relative Davenport constant of $\mathcal{S}$ with respect to $a$, denoted ${\rm D}_a(\mathcal{S})$, to be the largest $\ell\in \mathbb{N}\cup \{\infty\}$ such that there exists an irreducible sequence $T\in \mathcal{F}(\mathcal{S})$ with $|T|=\ell$ and $\sigma(T)=a$.}

Since any nonempty sequence $T\in \mathcal{F}(\mathcal{S})$ with $\sigma(T)=0_{\mathcal{S}}$ is reducible, we shall admit normally that  $${\rm D}_a(\mathcal{S})=0 \ \ \mbox{ if } \ \ a=0_{\mathcal{S}}.$$

In fact, due to the research of Factorization Theory in Algebra, A. Geroldinger and F. Halter-Koch \cite{GH} in 2006 have formulated another closely related definition, denoted ${\rm d}(\mathcal{S})$, for any commutative semigroup  $\mathcal{S}$, which is called the small Davenport constant.

\medskip

\noindent \textbf{Definition C.} (Definition 2.8.12 in \cite{GH}) \ {\sl For a commutative semigroup $\mathcal{S}$, let ${\rm d}(\mathcal{S})$ be the smallest $\ell \in
\mathbb{N}_0\cup \{\infty\}$ with the following property:

For any $m\in \mathbb{N}$ and $a_1, \ldots,a_m\in \mathcal{S}$ there
exists a subset $I\subset [1,m]$ such that $|I|\leq \ell$ and
$$
\sum_{i=1}^m a_i=\sum_{i \in I}a_i.
$$}

The following connection between the (large) Davenport constant ${\rm D}(\mathcal{S})$ and the small Davenport constant ${\rm d}(\mathcal{S})$ was found when $\mathcal{S}$ is a finite commutative semigroup.

\medskip

\noindent \textbf{Proposition D.}  \  {\sl Let $\mathcal{S}$ be a finite  commutative semigroup. Then,
\begin{enumerate}
\item  $\mathsf d(\mathcal{S})<
\infty.$ (see Proposition 2.8.13 in \cite{GH})
\item  $\mathsf {\rm D}(\mathcal{S})=\mathsf d(\mathcal{S})+1.$ (see Proposition 1.2 in \cite{AdhikariGaoWang14})
\end{enumerate}
}

Concerned with some specific kind of semigroups, the author of this manuscript together with W.D. Gao in 2008 obtained the following result.

\medskip

\noindent \textbf{Theorem E.} \cite{wanggao} \ {\sl Let
$R = \mathbb{Z}\diagup{n_1\mathbb{Z}}\oplus\cdots \oplus \mathbb{Z}\diagup{n_r\mathbb{Z}}$ be the direct sum of $r$ residual class rings modulo $n_1,\ldots,n_r$ respectively. Let ${\bf a}=(\overline{a_1},\ldots,\overline{a_r})$ be an element of $\mathcal{S}_R$, where $\overline{a_i}=a_i+n_i\mathbb{Z}\in \mathbb{Z}\diagup{n_i}\mathbb{Z}$ for $i\in [1,r]$. Let $R'=\mathbb{Z}\diagup{\frac{n_1}{t_1}\mathbb{Z}}\oplus\cdots \oplus \mathbb{Z}\diagup{\frac{n_r}{t_r}\mathbb{Z}}$, where $t_i=\gcd(a_i,n_i)$ for $i\in [1,r]$.  Then
$$\begin{array}{llll}{\rm D}_{\bf a}(\mathcal{S}_R)=\left\{\begin{array}{llll}
               {\rm D}_{\bf a}({\rm U}(R)),  & \mbox{if \ \ } a\in {\rm U}(R);\\
               \sum\limits_{i=1}^r \Omega(t_i)+{\rm D}({\rm U}(R'))-1 ,  & \mbox{if otherwise},\\
              \end{array}
              \right.
\end{array}$$
where $\Omega(t_i)$ denotes the number of prime factors (repeat prime factors are also calculated) of the integer $t_i$.}

In this manuscript, we shall make a study of the largest length of irreducible sequences representing any given element $a\in \mathcal{S}^{\bullet}$ in the setting of general commutative semigroup $\mathcal{S}$, and try to
bound ${\rm D}_a(\mathcal{S})$. Although in any abelian group $G$, as stated as in Theorem A, ${\rm D}_a(G)$ and ${\rm D}(G)$ we strongly related and bounded each other,
in the setting of commutative semigroups $\mathcal{S}$, one interesting and seemingly very natural thing is that ${\rm D}_a(\mathcal{S})$ and ${\rm D}(\mathcal{S})$ do not have this mutual restraint any more. For example, take $\mathcal{S}$ to be the additive semigroup $\mathbb{N}$. Then ${\rm D}_a(\mathcal{S})$ is finite for every $a\in \mathcal{S}$, but ${\rm D}(\mathcal{S})$ is infinite.
 Nevertheless, in this paper we still obtained that ${\rm D}_a(\mathcal{S})$ was bounded by some its `local structure' in the semigroup $\mathcal{S}$.

 While, we first give the following conclusion which can be derived directly from the definitions of ${\rm D}_a(\mathcal{S})$ and ${\rm D}(\mathcal{S})$.

\medskip

\begin{prop}\label{proposition Da(S) and D(S)} \ Let $\mathcal{S}$ be a commutative semigroup. Then,

\noindent (i). \  If $\mathcal{S}$ is a monoid and
$a\in {\rm U}(\mathcal{S})$, then
${\rm D}_a(\mathcal{S})={\rm D}_a({\rm U}(\mathcal{S}))$;

\noindent (ii). \ ${\rm D}(\mathcal{S})$ is finite if and only if ${\rm D}_a(\mathcal{S})$ is bounded for all $a\in \mathcal{S}$, i.e., there exists a given large integer $\mathcal{M}$ such that ${\rm D}_a(\mathcal{S})\leq \mathcal{M}$ for all $a\in \mathcal{S}$.
In particular, if ${\rm D}(\mathcal{S})$ is finite then $${\rm D}(\mathcal{S})=\max\limits_{a\in \mathcal{S}}\{{\rm D}_a(\mathcal{S})\}+1.$$
 \end{prop}

Another thing worth mentioning is that the research on additive properties of irreducible sequences seems to have a close connection with the Word Problem for groups and semigroups. Given a semigroup (or group)  $\mathcal{S}=<X\mid \mathcal{R}>$ generated by the set $X$ subject to the defining relations $\mathcal{R}$, and given any two words (finite, or infinite sequences), say ${\bf u}=u_1\cdot u_2\cdot \ldots\cdot u_{\ell}$ and ${\bf v}=v_1\cdot v_2\cdot \ldots\cdot v_{t}$, with $u_1,\ldots,u_{\ell}, v_1,\ldots,v_t\in X$, decide whether ${\bf u}\equiv {\bf v} \pmod {\mathcal{R}}$, i.e., whether ${\bf u}$ and ${\bf v}$ represent the same element of $\mathcal{S}$. Dehn \cite{Dehn} in
1911 investigated some special cases of the word problem
for groups. The question Dehn proposed was to find a uniform test or mechanical
procedure (i.e. an algorithm) which enables us to decide whether ${\bf u}$ and ${\bf v}$ represent the same element of $\mathcal{S}$. If there is such an algorithm, the word problem is called solvable. In 1947, Post \cite{Post} proved that the word problem for
semigroups is unsolvable. Later Novikov
\cite{Novikov} in 1955, and Boone \cite{Boone} in 1959,
showed that the word problem for groups is also unsolvable. To learn more on the word problem in commutative semigroups, one is refereed to \cite{wordAdvance,IsrealWord}. Though the computation complexity is the main consideration in the field of Word Problem, determining the largest length of irreducible words (which does not represent the same element as any its proper sub-word does) represent some given element in the semigroup $\mathcal{S}$ is still an interesting problem.

In this manuscript, the largest length of irreducible sequences representing any given element $a\in \mathcal{S}^{\bullet}$, i.e., the value of ${\rm D}_a(\mathcal{S})$, is investigated in commutative semigroups. For any element $a$ of a commutative semigroup $\mathcal{S}$, in case that any ascending chain of principal ideals starting from the ideal $(a)$ terminates in $\mathcal{S}$, we give the sufficient and necessary conditions to ensure that ${\rm D}_a(\mathcal{S})$ is finite, and moreover, we give the sharp lower and upper bounds of ${\rm D}_a(\mathcal{S})$ when ${\rm D}_a(\mathcal{S})$ is finite. We also applied the obtained result to commutative unitary rings, in particular, we determined the precise value of ${\rm D}_a(\mathcal{S})$ when $\mathcal{S}$ is the multiplicative semigroup of any finite commutative principal ideal unitary ring.

Before giving our main theorems, some necessary notations and terminologies in Semigroups  will be worth reviewing for readers who do not specialize in Semigroup Theory.

For any subset $A\subseteq \mathcal{S}$, let $${\rm St}(A)=\{c\in \mathcal{S}: c+a\in A \ \mbox{for every} \ a\in A\}$$ be the stabilizer of the set $A$ in the semigroup $\mathcal{S}$, which is a subsemigroup of $\mathcal{S}$.
For any element $a\in \mathcal{S}$, let
$$(a)=\{a+c: c\in \mathcal{S}^{0}\}$$ denotes the {\sl principal ideal} generated by the element $a\in \mathcal{S}$.
The {\bf Green's preorder} on the semigroup $\mathcal{S}$, denoted $\leqq_{\mathcal{H}}$, is defined by
$$a \leqq_{\mathcal{H}} b\Leftrightarrow a=b \ \ \mbox{or}\ \ a=b+c$$ for some $c\in \mathcal{S}$, equivalently, $$(a)\subseteq (b).$$ Green's congruence, denoted
$\mathcal{H}$, is a basic relation introduced by Green for semigroups which is defined by:
$$a \ \mathcal{H} \ b \Leftrightarrow a \ \leqq_{\mathcal{H}} \ b \mbox{ and } b \ \leqq_{\mathcal{H}} \ a \Leftrightarrow (a)=(b).$$
For any element $a$ of $\mathcal{S}$,  let $H_a$ be the congruence class by $\mathcal{H}$ containing $a$.  We write $a\prec_{\mathcal{H}} b$ to mean that $a \leqq_{\mathcal{H}} b$ but $H_a\neq H_b$.

Let $\Lambda$ be any partially ordered set. We say $\Lambda$ has the ascending chain condition (a.c.c.), provided that any ascending chain $\lambda_1<\lambda_2<\cdots$ terminates. We say a commutative semigroup $\mathcal{S}$ (a commutative ring $R$) satisfies a.c.c. for principal ideals, or for ideals, or for congruences provided that the above corresponding partially ordered set $\Lambda$ has the a.c.c., where $\Lambda$ denotes the partially ordered set consisting of  principal ideals, or of ideals, or of congruences, in $\mathcal{S}$ (in $R$) formed by inclusions, respectively. A  commutative semigroup $\mathcal{S}$ is said to be
{\bf Noetherian} provided that the semigroup $\mathcal{S}$ satisfies the a.c.c. for congruences.

Let $a$ be an element of any commutative semigroup $\mathcal{S}$, or an element of any commutative unitary ring $R$. We define $\Psi(a)$ to be the largest length $\ell\in \mathbb{N}_0\cup \{\infty\}$ of strictly ascending principal ideals chain of $\mathcal{S}^{0}$ (of $R$ accordingly) starting from $(a)$, i.e., the largest $\ell\in \mathbb{N}_0\cup \{\infty\}$ such that there exist $\ell$ elements $a_1,a_2,\ldots,a_{\ell}\in \mathcal{S}^{0}$ ($a_1,a_2,\ldots,a_{\ell}\in R$ respectively) with $$(a) \subsetneq  (a_1) \subsetneq \cdots \subsetneq (a_{\ell}).$$ When $a\in \mathcal{S}$, $\Psi(a)$ can be equivalently defined as
the largest length of strictly ascending Green's preorder chain starting from $a$:
$$a \prec_{\mathcal{H}} a_1\prec_{\mathcal{H}}\cdots \prec_{\mathcal{H}} a_{\ell}.$$

Since the commutative ring $R$ is unitary, we see that all the  principal ideals of the ring $R$ are consistent with all the principal ideals of the semigroup $\mathcal{S}_R$, and therefore, the definition $\Psi(a)$ is consistent no matter whether we regard $a$ as the element of the ring $R$ or as the element of the multiplicative semigroup $\mathcal{S}_R$.

Then we introduce the definition of {\bf Sch${\rm \ddot{u}}$tzenberger group} which palys a key role in giving the bounds for ${\rm D}_a(\mathcal{S})$.

Each $c\in {\rm St}(H_a)$ induces a mapping $$\gamma_c:H_a\rightarrow H_a$$ defined by
$$\gamma_c: x\mapsto c+x$$ for every $x\in H_a$, which we write as a operator $$\gamma_c \circ x=c+x.$$ Let $$\Gamma(H_a)=\{\gamma_c: c\in {\rm St}(H_a)\}.$$
It is well known that $\Gamma(H_a)$ is an abelian group with the operation
\begin{equation}\label{equation operation}
\gamma_c+\gamma_d=\gamma_{c+d},
\end{equation} which is discover by M.P. Sch${\rm \ddot{u}}$tzenberger  in 1957 (see  Section 3 of Chapter II in \cite{Grillet Semigroup}), and is called the Sch${\rm \ddot{u}}$tzenberger group of $H_a$. Naturally, there exists a homomorphism $\rho_a$ of ${\rm St}(H_a)$ onto $\Gamma(H_a)$, defined by $$\rho_a: c\mapsto \gamma_c$$ for every $c\in {\rm St}(H_a)$.

Now we are in a position to put out our main results of this manuscript, which are Theorem \ref{Theorem Main theorem abstract} and Theorem \ref{Theorem rings}.

\bigskip

\begin{theorem}\label{Theorem Main theorem abstract} \ Let $\mathcal{S}$ be a commutative semigroup. Let $a$ be an element of $\mathcal{S}^{\bullet}$ with $\Psi(a)$ being finite. If $|H_a|$ is infinite then ${\rm D}_a(\mathcal{S})$ is infinite, and if $|H_a|$ is finite then ${\rm D}_a(\mathcal{S})$ is finite and
$$\epsilon \ {\rm D}(\Gamma(H_a))\leq {\rm D}_a(\mathcal{S})\leq \Psi(a)+{\rm D}(\Gamma(H_a))-1$$
where
$$\begin{array}{llll}\epsilon=\left\{\begin{array}{llll}
               \frac{1}{2},  & \mbox{if \ \ } {(a+a) \ \mathcal{H} \ a;}\\
              1,  & \mbox{if \ \ } otherwise,\
              \end{array}
              \right.
\end{array}$$  and both the lower and upper bounds are sharp.
\end{theorem}

\bigskip

\begin{theorem}\label{Theorem rings} \ Let $R$ be a  commutative unitary ring. Let $a$ be an element of $\mathcal{S}_R^{\bullet}$ with $\Psi(a)$ being finite.
Then  $$\Gamma(H_a)\cong {\rm U}(R_a),$$ where $R_a=R\diagup {\rm Ann}(a)$ be the quotient ring of $R$ modulo the annihilator of $a$. If ${\rm U}(R_a)$ is infinite then ${\rm D}_a(\mathcal{S}_R)$ is infinite, and if ${\rm U}(R_a)$ is finite then ${\rm D}_a(\mathcal{S}_R)$ is finite and
$$\epsilon \ {\rm D}({\rm U}(R_a))\leq {\rm D}_a(\mathcal{S}_R)\leq \Psi(a)+{\rm D}({\rm U}(R_a))-1,$$
where $\epsilon$ is the same as in Theorem \ref{Theorem Main theorem abstract}.  In particular,
if $R$ is a finite commutative principal ideal unitary ring and $a\notin {\rm U}(R)$, then the above equality  $${\rm D}_a(\mathcal{S}_R)= \Psi(a)+{\rm D}({\rm U}(R_a))-1$$ holds.
\end{theorem}

Note that
a commutative semigroup $\mathcal{S}$  satisfies the a.c.c. for principal ideals if and only if $\Psi(a)$ is finite for all element $a\in \mathcal{S}$. Hence, we have the following corollary.

\medskip

\begin{cor} \label{corollary in a.c.c.}\ Let $\mathcal{S}$ be a commutative semigroup satisfying the a.c.c. for principal ideals, and let $a$ be an element of $\mathcal{S}^{\bullet}$. If $|H_a|$ is infinite then ${\rm D}_a(\mathcal{S})$ is infinite, and if $|H_a|$ is finite then ${\rm D}_a(\mathcal{S})$ is finite and
$$\epsilon \ {\rm D}(\Gamma(H_a))\leq {\rm D}_a(\mathcal{S})\leq \Psi(a)+{\rm D}(\Gamma(H_a))-1$$
where $\epsilon$ is the same as in Theorem \ref{Theorem Main theorem abstract}.
\end{cor}

It is worth remarking that the semigroup $\mathcal{S}$ in Corollary \ref{corollary in a.c.c.} is more general than commutative Noetherian semigroups. In precise, any commutative Noetherian semigroup $\mathcal{S}$ must satisfy the a.c.c. for principal ideals,
while the converse is not necessarily true. It was proved that (see Theorem 5.1 of Chapter I in \cite{SemigroupRing}) any commutative Noetherian semigroup must satisfy the a.c.c. for ideals, and therefore, must satisfy the a.c.c. for principal ideals. Conversely, the free commutative semigroup $F_{\mathbb{N}}$ generated by $\mathbb{N}$ satisfies the a.c.c. for principal ideals but is not Noetherian.

\medskip

In  Corollary \ref{corollary in a.c.c.}, if the element $a$ is an unit, then $\Gamma(H_a)\cong H_a={\rm U}(\mathcal{S})$ (see Lemma \ref{Lemma schutzenburger group}) and $\Psi(a)=0$. Hence, we have the following corollary.

\medskip

\begin{cor} \label{corollary U(S)}\ Let $\mathcal{S}$ be a commutative monoid, and let $a\in {\rm U}(\mathcal{S})\setminus \{0_{\mathcal{S}}\}$. If ${\rm U}(\mathcal{S})$ is infinite then ${\rm D}_a(\mathcal{S})$ is infinite, and if ${\rm U}(\mathcal{S})$ is finite then ${\rm D}_a(\mathcal{S})$ is finite and
$$\frac{1}{2} \ {\rm D}({\rm U}(\mathcal{S}))\leq {\rm D}_a(\mathcal{S})\leq {\rm D}({\rm U}(\mathcal{S}))-1.$$
\end{cor}

\medskip

\begin{remark} \ In Corollary \ref{corollary U(S)}, if the monoid $\mathcal{S}$ is a finite abelian group, i.e., ${\rm U}(\mathcal{S})=\mathcal{S}$, then the conclusion of Corollary \ref{corollary U(S)} reduced to be Theorem A.
\end{remark}

\medskip

\begin{remark} \ It is not hard to see that in Theorem E, the ring $R = \mathbb{Z}\diagup{n_1\mathbb{Z}}\oplus\cdots \oplus \mathbb{Z}\diagup{n_r\mathbb{Z}}$ is a finite principal ideal unitary ring, moreover, $\Psi({\bf a})=\sum\limits_{i=1}^r\Omega(t_i)$ and $R'=R\diagup {\rm Ann}(a)$. That is, Theorem E is a corollary of Theorem \ref{Theorem rings} within the case of finite principal ideal unitary rings.
\end{remark}

\section{Proofs of Theorem 1.2 and Theorem 1.3}

We begin this section with some necessary lemmas.

\begin{lemma} \label{Lemma folklore davenport} (folklore) \ Let $G$ be a finite abelian group and let $T\in \mathcal{F}(G)$ be a minimal zero-sum sequence of length ${\rm D}(G)$. Then $$\Sigma(T)=G,$$ where $\Sigma(T)=\{\sigma(V):V\mid T\mbox{ and }T\neq \varepsilon\}$ is the set consisting of all elements of $G$ that can be represented a sum of some terms from $T$.
\end{lemma}

\medskip

\begin{lemma} \label{Lemma find a sequence} \ Let $G$ be an abelian group, and let $g$ be an element of $G$. Then there exists a sequence $T\in \mathcal{F}(G)$ of sum $\sigma(T)=g$ and of length
$$\begin{array}{llll}|T|\geq \left\{\begin{array}{llll}
               \frac{1}{2} {\rm D}(G),  & \mbox{if \ \ } {|G| \mbox{ is finite};}\\
               \mathcal{M},  & \mbox{if \ \ } otherwise,\
              \end{array}
              \right.
\end{array}$$
such that $T$ contains no nonempty proper zero-sum subsequence, where $\mathcal{M}$ denotes any given positive integer.
\end{lemma}

\begin{proof} We consider first the case that $|G|$ is infinite. Let $L\in \mathcal{F}(G)$ be an arbitrary  zero-sum free sequence such that $g\notin \sum(L)$. Since $|G|$ is infinite and $|\Sigma(L)\cup \{0_G\}|$ is finite, there exists some element $b\in G\setminus \{0_G\}$ such that $0_G, g\notin b+(\Sigma(L)\cup \{0_G\})=\Sigma(L\cdot b).$ By the arbitrariness of $L$, we can find a sequence $V\in \mathcal{F}(G)$ inductively with length at least $\mathcal{M}$ and
\begin{equation}\label{equation find a sequence 1}
0_G, g\notin \Sigma(V).
\end{equation}
Let \begin{equation}\label{equation find a sequence 2}
g'=g-\sigma(V),\end{equation} and let \begin{equation}\label{equation find a sequence 3}
T=g'\cdot V.\end{equation}
By \eqref{equation find a sequence 1}, \eqref{equation find a sequence 2} and \eqref{equation find a sequence 3}, we can verify that $T$ is a sequence of length $|T|>\mathcal{M}$ and of sum $\sigma(T)=g$ such that $T$ contains no nonempty proper zero-sum subsequence.

The remaining case that $|G|$ is finite and $g\neq 0_G$ follows immediately from Theorem A given by Ska{\l}ba \cite{Skalba}. For the reader's convenience, we present its short proof below.

Assume $|G|$ is finite. Take a minimal zero-sum sequence $L\in \mathcal{F}(G)$ of length $|L|={\rm D}(G)$. By Lemma \ref{Lemma folklore davenport}, there exists a nonempty subsequence $V$ with $\sigma(V)=g$. If $|V|\geq \frac{1}{2}{\rm D}(G)$, we are done. Hence, we assume $|V|<\frac{1}{2}{\rm D}(G)$. We check that $$T=\prod\limits_{c\mid LV^{[-1]}} (-c)$$ is a sequence with $$|\prod\limits_{c\mid LV^{[-1]}} (-c)|=|LV^{[-1]}|>\frac{1}{2}{\rm D}(G)$$ and  $$\sigma(\prod\limits_{c\mid LV^{[-1]}}(-c))=-\sigma(LV^{[-1]})=\sigma(V)=g,$$ which contains no nonempty proper zero-sum subsequence.
This proves the lemma.
\end{proof}

\medskip

\begin{lemma}\cite{SavChen,Yuanpingzhi}\label{Lemma savchen} \ For $n>1$, let $\mathbb{Z}_n$ be the additive group of integers modulo $n$.  Let $T\in \mathcal{F}(\mathbb{Z}_n)$ be a zero-sum free sequence of length greater than $\frac{n}{2}$. Then there exists some integer $b$ coprime to $n$ such that $\sum\limits_{c\mid T}|b c|_n<n$, where $|b c|_n$ denotes the least positive residue of $bc$ modulo $n$.
\end{lemma}

\medskip

\begin{lemma}\label{Lemma Da(Zn)} \  Let $n>1$ be an even number. Let $a$ be the unique element of order two in the group $\mathbb{Z}_n$. Then ${\rm D}_a(\mathbb{Z}_n)=\frac{n}{2}$.
\end{lemma}

\begin{proof} We see $\bar{1}^{[\frac{n}{2}]}\in \mathcal{F}(\mathbb{Z}_n)$ is an irreducible sequence of length $\frac{n}{2}$ and of sum $a$. This implies that
${\rm D}_a(\mathbb{Z}_n)\geq \frac{n}{2}$. To prove the conclusion, we suppose to the contrary that ${\rm D}_a(\mathbb{Z}_n)>\frac{n}{2},$ i.e., there exists an irreducible sequence $T\in \mathcal{F}(\mathbb{Z}_n)$ of length at least $\frac{n}{2}+1$ with $\sigma(T)=a$. Since $a\neq 0_G$, we see that $T$ is also zero-sum free. By Lemma \ref{Lemma savchen}, we may assume w.l.o.g. that $$\sum\limits_{c\mid T} |c|_n<n.$$ Since $|T|>\frac{n}{2}$, we also have $$\sum\limits_{c\mid T} |c|_n>\frac{n}{2},$$ a contradiction with ${\rm ord}(\sigma(T))=2$. This proves the lemma.
\end{proof}

\medskip

\begin{lemma}\label{Lemma folklore} (folklore) \ For any element $a\in \mathcal{S}^{0}$, ${\rm U}(\mathcal{S}^{0})$ acts on the congruence class  $H_a$.
\end{lemma}

\medskip

\begin{lemma} (see Lemma 4.2, Proposition 4.3 and Proposition 4.6 of Chapter I in \cite{Grillet monograph})\label{Lemma schutzenburger group} \  Let $a$ be an element of $\mathcal{S}$. Then,

(i). the Sch$\ddot{u}$tzenberger group $\Gamma(H_a)$ is a simply transitive group of permutations of $H_a$;

(ii). $H_a$ is a subgroup of $\mathcal{S}$ if and only if $(a+a) \ \mathcal{H} \ a$;

(iii). if $H_a$ is a subgroup of $\mathcal{S}$ then $\Gamma(H_a)\cong H_a$.
\end{lemma}

\medskip

Now we are in a position to prove Theorem \ref{Theorem Main theorem abstract}.

\bigskip

\noindent {\sl Proof of Theorem \ref{Theorem Main theorem abstract}}. \
By Conclusion (i) of Lemma \ref{Lemma schutzenburger group}, we conclude that there exists a bijection of $\Gamma(H_a)$ onto $H_a$, i.e., $$|\Gamma(H_a)|=|H_a|.$$
Also, it is well known that $|\Gamma(H_a)|$ is finite if and only if ${\rm D}(\Gamma(H_a))$ is finite.

We first assume that $|H_a|$ is finite and prove ${\rm D}_a(\mathcal{S}_R)\leq \Psi(a)+{\rm D}(\Gamma(H_a))-1$. Take an arbitrary sequence $T\in {\cal F}(\mathcal{S}_R)$ with  \begin{equation}\label{equation |T|geq}
|T|\geq \Psi(a)+{\rm D}(\Gamma(H_a))
\end{equation} and $$\sigma(T)=a.$$ It suffices to show that the sequence $T$ is reducible. Let $T_1$ be a shortest subsequence of $T$ with
\begin{equation}\label{equation sigma(T1)Hsigma(T)}
\sigma(T_1) \ \mathcal{H} \ \sigma(T),
\end{equation}
 (note that $T_1$ is perhaps the empty subsequence $\varepsilon$ when $\sigma(T)\in {\rm U}(\mathcal{S})$).
 Assume $$T_1=\prod\limits_{i=1}^{k} a_i$$ where $k=|T_1|\geq 0$. By the minimality of $|T_1|$, we have that $$0_{\mathcal{S}}\succ_{\mathcal{H}} a_1\succ_{\mathcal{H}} a_1+a_2\succ_{\mathcal{H}}\cdots \succ_{\mathcal{H}}\sum\limits_{i=1}^k a_i \ \mathcal{H}\ a,$$ which implies that
\begin{equation}\label{equation length T1}
|T_1|=k\leq \Psi(a).
\end{equation}

For each term $c\mid TT_1^{[-1]}$, since $\sigma(T)=\sigma(T_1)+c+\sigma(TT_1^{[-1]} c^{[-1]})\leqq_{\mathcal{H}} \sigma(T_1)+c\leqq_{\mathcal{H}} \sigma(T_1)$, combined with \eqref{equation sigma(T1)Hsigma(T)}, we have that $$(\sigma(T)+c) \ \mathcal{H} \ (\sigma(T_1)+c) \ \mathcal{H} \ \sigma(T)$$ and thus $$c\in {\rm St}(H_a).$$ Combined with \eqref{equation |T|geq} and \eqref{equation length T1}, we see that $$\prod\limits_{c\mid TT_1^{[-1]}} \gamma_c\in {\cal F}(\Gamma(H_a))$$ is a sequence of length $|\prod\limits_{c\mid TT_1^{[-1]}} \gamma_c|=|TT_1^{[-1]}|=|T|-|T_1|\geq {\rm D}(\Gamma(H_a)).$ Combined with \eqref{equation operation},  there exists a {\bf nonempty} subsequence $T_2\mid TT_1^{[-1]}$ such that
\begin{equation}\label{equation sum gamma=0}
\gamma_{\sigma(T_2)}=\sigma(\prod\limits_{c\mid T_2}\gamma_c)=0_{\Gamma(H_a)}.
\end{equation}
By \eqref{equation sigma(T1)Hsigma(T)} and \eqref{equation sum gamma=0}, we conclude that
$$\begin{array}{llll}
\sigma(T)&=& \sigma(TT_1^{[-1]}T_2^{[-1]})+(\sigma(T_1)+\sigma(T_2)) \\
&=& \sigma(TT_1^{[-1]}T_2^{[-1]})+\gamma_{\sigma(T_2)}\circ \sigma(T_1)\\
&=&  \sigma(TT_1^{[-1]}T_2^{[-1]})+0_{\Gamma(H_a)}\circ\sigma(T_1) \\
&=& \sigma(TT_1^{[-1]}T_2^{[-1]})+\sigma(T_1) \\
&=& \sigma(TT_2^{[-1]}),\\
\end{array}$$
which implies $T$ is reducible. Hence, for the case $|H_a|$ is finite, ${\rm D}_a(\mathcal{S})$ is finite and $${\rm D}_a(\mathcal{S})\leq \Psi(a)+{\rm D}(\Gamma(H_a))-1$$ is proved.

\medskip

Now we assume that $(a+a) \ \mathcal{H} \ a$.
By Conclusions (ii) and (iii) of Lemma \ref{Lemma schutzenburger group}, $H_a$ is a subgroup of the semigroup $\mathcal{S}$ and
\begin{equation}\label{equation D(Ha)=D(Gamma)}
\Gamma(H_a)\cong H_a.
\end{equation}
By \eqref{equation D(Ha)=D(Gamma)} and Lemma \ref{Lemma find a sequence}, since $a$ is not the identity element of the semigroup $\mathcal{S}$ (even if $\mathcal{S}$ has an identity),
we can find a sequence $T\in \mathcal{F}(H_a)$ with $$\sigma(T)=a$$ and
$$\begin{array}{llll}|T|\geq \left\{\begin{array}{llll}
               \frac{1}{2} \ {\rm D}(\Gamma(H_a)),  & \mbox{if \ \ } {|H_a| \mbox{ is finite};}\\
               \mathcal{M},  & \mbox{if \ \ } otherwise,\
              \end{array}
              \right.
\end{array}$$
such that $T$ contains no {\bf proper} subsequence $T'$ with $\sigma(T')=\sigma(T)$,  i.e., $T$ is irreducible. This proves that for the case that $(a+a) \ \mathcal{H} \ a$, if $H_a$ is infinite then ${\rm D}_a(\mathcal{S})$ is infinite, and if $H_a$ is finite then $${\rm D}_a(\mathcal{S})\geq \frac{1}{2} \ {\rm D}(\Gamma(H_a)).$$

Now assume that $(a+a) \ \mathcal{H} \ a$ does not hold.
By Lemma \ref{Lemma find a sequence}, we can take a sequence $$V\in {\cal F}({\rm St}(H_a))$$ of length
\begin{equation}\label{equation |V|=two choices}
\begin{array}{llll}|V|=\left\{\begin{array}{llll}
               {\rm D}(\Gamma(H_a))-1,  & \mbox{if \ \ } {|H_a| \mbox{ is finite};}\\
               \mathcal{M},  & \mbox{if \ \ } otherwise,\
              \end{array}
              \right.
\end{array}
\end{equation}
 such that $\prod\limits_{c\mid V} \gamma_c\in {\cal F}(\Gamma(H_a))$ is a zero-sum free sequence in the group $\Gamma(H_a)$. Since $\sigma(\prod\limits_{c\mid V} \gamma_c)\neq 0_{\Gamma(H_a)}$, it follows from Conclusion (i) of Lemma \ref{Lemma schutzenburger group} that there exists some element
\begin{equation}\label{equation b in Ha}
b\in H_a\setminus \{a\}
\end{equation} with
\begin{equation}\label{equation sigma(V)+b=a}
\sigma(V)+b= \gamma_{\sigma(V)} \circ b=\sigma(\prod\limits_{c\mid V} \gamma_c) \circ b=a.
\end{equation}
We need to show that $V\cdot b$ is irreducible.
Assume to the contrary that $V\cdot b$ contains a {\bf proper} subsequence $W$ such that $$\sigma(W)=a.$$
If $b\not \mid W$, then $W\mid V$, combined with \eqref{equation b in Ha}, we have that
$$\begin{array}{llll}
a&=& \sigma(V\cdot b)\\
&=&\sigma(W)+\sigma(VW^{[-1]})+b \\
&\leqq_{\mathcal{H}}& \sigma(W)+b \\
&=& a+b \\
&\mathcal{H}& \ a+a \\
&\prec_{\mathcal{H}}& a,\\
\end{array}$$ which is absurd. Hence, we have $$b\mid W.$$ Then $Wb^{[-1]}$ is a {\bf proper} subsequence of $V$. Recalling that the sequence $\prod\limits_{c\mid V} \gamma_c$ is zero-sum free in the group $\Gamma(H_a)$,  we have that $$\gamma_{\sigma(Wb^{[-1]})}\neq \gamma_{\sigma(V)}.$$ By Conclusion (i) of Lemma \ref{Lemma schutzenburger group}, we have $a=\sigma(W)=\sigma(Wb^{[-1]})+b=\gamma_{\sigma(Wb^{[-1]})}\circ b\neq \gamma_{\sigma(V)} \circ b=a$, which is absurd too.
Therefore, this proves that $V\cdot b$ is irreducible. Combined with \eqref{equation |V|=two choices} and \eqref{equation sigma(V)+b=a}, we derive that
$$\begin{array}{llll}{\rm D}_a(\mathcal{S})\geq |V\cdot b|=\left\{\begin{array}{llll}
               {\rm D}(\Gamma(H_a)),  & \mbox{if \ \ } {|H_a| \mbox{ is finite};}\\
               \mathcal{M}+1,  & \mbox{if \ \ } otherwise,\
              \end{array}
              \right.
\end{array}
$$
for the case that $(a+a) \ \mathcal{H} \ a$ does not hold.

\medskip

Now it remains to show that the upper and lower bounds are sharp. The sharpness of the upper bound will be given by the conclusion in Theorem \ref{Theorem rings}. We shall give examples to show the lower bounds are sharp in the rest arguments of this theorem.

Take a positive even integer $n$.
Let $\mathcal{S}_1=<X\mid \mathscr{R}>$ be a finite commutative semigroup generated by the set $X=\{x_1,\ldots,x_r\}$ subject to the defining relation $\mathscr{R}$, where \begin{equation}\label{equation R1=}
\mathscr{R}=\{(n+1) x_i=x_i:i=1,2,\ldots,r\}\cup \{x_i+x_j=x_j:1\leq i<j\leq r\}.
\end{equation}
Take $$a=\frac{n}{2} x_k
 \mbox{ for some }k\in \{1,2,\ldots,r\}.$$ By  \eqref{equation R1=}, we see that
 \begin{equation}\label{equation (a+a)Ha}
 (a+a) \ \mathcal{H} \ a
 \end{equation}
  and
  \begin{equation}\label{equation Ha=Zn}
  H_a=\langle x_k\rangle\cong\mathbb{Z}_n,
   \end{equation} and that
any irreducible sequence $L\in \mathcal{F}(\mathcal{S}_1)$ with $\sigma(L)=a$ must be a sequence of all terms from the subgroup $\langle x_k\rangle$. Combined with \eqref{equation D(Ha)=D(Gamma)}, \eqref{equation (a+a)Ha}, \eqref{equation Ha=Zn}, Lemma \ref{Lemma Da(Zn)} and Conclusion (iii) of Lemma \ref{Lemma schutzenburger group}, we conclude that $${\rm D}_a(\mathcal{S})=\frac{n}{2}=\frac{1}{2}{\rm D}(\mathbb{Z}_n)=\frac{1}{2}{\rm D}(H_a)=\frac{1}{2}{\rm D}(\Gamma(H_a)).$$ This proves that for the case that $(a+a)\ \mathcal{H}\  a$, the lower bound  ${\rm D}_a(\mathcal{S})\geq \frac{1}{2} {\rm D}(\Gamma(H_a))$ is sharp.

 \medskip

 Take an integer $m>2$. Let $\mathcal{S}_2=<X\mid \mathscr{R}>$ be a finite commutative semigroup with a zero element $\infty$ generated by the set $X=\{x_0,x_1,\ldots,x_m\}$ subject to the defining relation $\mathscr{R}$, where
\begin{equation}\label{equation R2=}
\mathscr{R}=\{(m+1) x_0=x_0\}\cup \{x_i+x_j=\infty: i,j\in [1,r]\}\cup \{x_0+x_k=x_{|k+1|_m}: k=1,2,\ldots,m\},
\end{equation}
 and $|k+1|_m$ denotes the least positive residue of $k+1$ modulo $m$. Take $$a=x_m.$$ We see that
 \begin{equation}\label{equation Ha consisting of}
 H_a=\{x_1,\ldots,x_m\},
  \end{equation}
 \begin{equation}\label{equation a+a<a}
 (a+a)\prec_{\mathcal{H}} a,
 \end{equation}
 and
  $$\Gamma(H_a)\cong \langle x_0\rangle \cong \mathbb{Z}_m.$$
Take an arbitrary sequence $L\in \mathcal{F}(S_2)$ with
\begin{equation}\label{equation |L|geq m+1}
|L|\geq {\rm D}(\Gamma(H_a))+1=m+1
\end{equation}
 and \begin{equation}\label{equation sigma(L)=a}
 \sigma(L)=a.
\end{equation}
  By \eqref{equation R2=} and \eqref{equation Ha consisting of}, we derive that $L$ contains exactly one term from $H_a$, and so at least $m$ terms from the subgroup $\langle x_0\rangle$.  Since ${\rm D}(\langle x_0\rangle)={\rm D}(\mathbb{Z}_m)=m$, it follows that there exists a {\bf nonempty} subsequence $L_1$ of $L$ with $\sigma(L_1)=m x_0$, the identity element of the group $\langle x_0\rangle$ which is also the identity element of the semigroup $\mathcal{S}_2$. Hence, $\sigma(L)=\sigma(L_1)+\sigma(LL_1^{[-1]})=mx_0+\sigma(LL_1^{[-1]})=\sigma(LL_1^{[-1]})$, and so $L$ is reducible. By the arbitrariness of $L$ and \eqref{equation |L|geq m+1} and \eqref{equation sigma(L)=a}, we have that
  $${\rm D}_a(\mathcal{S}_2)\leq m={\rm D}(\Gamma(H_a)).$$
  By \eqref{equation a+a<a}, we have ${\rm D}_a(\mathcal{S}_2)\geq  {\rm D}(\Gamma(H_a))$ and so $${\rm D}_a(\mathcal{S}_2)={\rm D}(\Gamma(H_a)).$$
  This proves the lower bound ${\rm D}_a(\mathcal{S})\geq {\rm D}(\Gamma(H_a))$ is sharp for the case that $(a+a)\ \mathcal{H} \ a$ does not hold, and therefore, completes the proof of the Theorem. \qed

\bigskip

To prove Theorem \ref{Theorem rings}, several preliminaries will be necessary.

In the rest of this section, we always admit that $R$ is a commutative unitary ring. The Green's congruence (preorder) are concerned with the operation in the semigroups $\mathcal{S}_R$, i.e., concerned with the multiplication operation of the ring $R$.
To avoid confusions with the previous usage in the proof of Theorem \ref{Theorem Main theorem abstract}, we need to fix some notations.
We still use $\cdot$ and $\prod$ to denote the concatenations of sequences. By $*$, $+_{R}$ and $-_{R}$ we denote the multiplication, addition and subtraction operations in the ring $R$,  respectively. For any sequence $T=\prod\limits_{i=1}^t a_i\in \mathcal{F}(\mathcal{S}_R)$, we denote $$\pi(T)=a_1*\cdots *a_t$$ to be the multiplications of all terms from $T$.

Let $K$ be an ideal of $R$. For any $i\in \mathbb{N}_0$, $K^i$ is the $i$-th power of the ideal $K$. In particular, $$K^{0}=R.$$ Define the index of the ideal $K$, denoted ${\rm ind}(K)$, to be the least $n\in \mathbb{N}_0\cup \{\infty\}$ such that $$K^n=K^{n+1},$$ equivalently, the descending chain of ideals $$K^{0}\supsetneq K^{1}\supsetneq  \cdots \supsetneq K^n=K^{n+1}=K^{n+2}=\cdots$$ keeps stationary starting from $K^n$. For any element $c\in R$,
 we define $\zeta(K:c)$ to be the largest $$\begin{array}{llll}t\in\left\{\begin{array}{llll}
               [0,{\rm ind}(K)],  & \mbox{if \ \ } {\rm ind}(K) \mbox{ is finite};\\
               \mathbb{N}_0\cup \{\infty\},  & \mbox{if \ \ }  \mbox{ otherwise,}\\
              \end{array}
              \right.
\end{array}$$
such that $$c\in K^{t}.$$ In particular, when $K=(a)$ is a principal ideal, we shall write ${\rm ind}(a)$ and $\zeta(a:c)$ in place of ${\rm ind}(K)$ and $\zeta(K:c)$, respectively.

\medskip

\begin{lemma}\label{lemma semigroup isomorphism theorem} (See Proposition 2.4 in Chapter I of \cite{Grillet monograph}) \ Let $\varphi: S\rightarrow T$ and $\tau: S\rightarrow U$ be homomorphisms of semigroups. If $\varphi$ is surjective, then $\tau$ factors through $\varphi$ ($\tau=\xi\circ \varphi$ for some homomorphism $\xi: T\rightarrow U$) if and only if ${\rm ker} \ \varphi\subseteq {\rm ker} \ \tau$; and then $\tau$
factors uniquely through $\varphi$ ($\xi$ is unique). If $\varphi$ and $\tau$ are surjective and ${\rm ker} \ \varphi={\rm ker} \ \tau$, then $\xi$ is an isomorphism.
 \end{lemma}

\medskip

 \begin{lemma}\label{Lemma structure of principal} \ Let $P$ be a finite commutative principal ideal unitary ring, and let $(a_1),\ldots,(a_r)$ be all the distinct maximal principal ideals of $P$. Let $b,c$ be elements of $P$. Then the following conclusions hold:

 (i). there exists a factorization $b=a_1^{\zeta(a_1:b)}* \cdots  * a_r^{\zeta(a_r:b)} * u$ for some unit $u\in {\rm U}(P)$;

 (ii). if there exist some $(m_1,\ldots, m_r)\in [0,{\rm ind}(a_1)]\times\cdots\times [0,{\rm ind}(a_r)]$ and $u\in {\rm U}(P)$ such that
 $b=a_1^{m_1}* \cdots * a_r^{m_r} * u$, then $(m_1,\ldots, m_r)=(\zeta(a_1:b),\ldots,\zeta(a_r:b))$;

 (iii). $(b)\subseteq (c)\Leftrightarrow  \zeta(a_i:b)\geq\zeta(a_i:c)\mbox{ for all }i\in [1,r]$;

 (iv). $(b)=(c)\Leftrightarrow\zeta(a_i:b)=\zeta(a_i:c)\mbox{ for all }  i\in [1,r]\Leftrightarrow b=c*u\mbox{ for some } u\in {\rm U}(P).$
 \end{lemma}

 \begin{proof} \  Since $P$ is finite, we have that ${\rm ind}(a_i)$ is finite, and therefore, $\zeta(a_i:b)$ is finite, where $i\in [1,r]$.
 Let $$n_i={\rm ind}(a_i)\ \mbox{ for }\ i=1,2,\ldots,r.$$

 \medskip

 (i). \ Let $$x=a_1^{\zeta(a_1:b)}* \cdots * a_r^{\zeta(a_r:b)}.$$
 Since $(a_1)^{\zeta(a_1:b)}, \ldots,(a_r)^{\zeta(a_r:b)}$ are coprime in pairs,
 we have that $$b\in \bigcap\limits_{i=1}^r(a_i)^{\zeta(a_i:b)}=(a_1)^{\zeta(a_1:b)}* \cdots * (a_r)^{\zeta(a_r:b)}=(a_1^{\zeta(a_1:b)}* \cdots * a_r^{\zeta(a_r:b)}),$$ which implies that $$b=x * d$$ for some element $d\in P$.
 Let $$I=\{i\in [1,r]: \zeta(a_i:b)<n_i\}.$$
 Note that
 \begin{equation}\label{equation d notin for i in I}
 d\notin (a_i)\mbox{ for each }i\in I.
 \end{equation}
 Since $(a_1)^{n_1}, \ldots,(a_r)^{n_r}$ are coprime in pairs, by the Chinese Remainder Theorem, we can find an element $\widetilde{d}\in P$ such that
 \begin{equation}\label{equation wide d=d}
 \widetilde{d}\equiv d \pmod {(a_i)^{n_i}}\mbox{ for each }i\in I
 \end{equation} and
 \begin{equation}\label{equation wide d=1}
 \widetilde{d}\equiv 1_P \pmod {(a_j)^{n_j}}\mbox{ for each }j\notin I.
 \end{equation}
  Combined \eqref{equation d notin for i in I}, \eqref{equation wide d=d} and \eqref{equation wide d=1}, we have that $\widetilde{d}\notin (a_i)$ for all $i\in [1,r]$, equivalently,
\begin{equation}\label{equation wide d in U(R)}
\widetilde{d}\in {\rm U}(P).
\end{equation}
 By \eqref{equation wide d=d} and \eqref{equation wide d=1}, we have that
 \begin{equation}\label{equation x wide equiv x d 1}
 x * \widetilde{d}\equiv x * d \pmod {(a_i)^{n_i}}\mbox{ for each }i\in I,
 \end{equation}
  and that
 \begin{equation}\label{equation x wide equiv x d 2}x * \widetilde{d}\equiv 0_P * \widetilde{d}=0_P=0_P * d\equiv x * d \pmod {(a_j)^{n_j}}\mbox{ for each } j\notin I.
  \end{equation}
 Since $P$ is Artinian, we know that the Jacobson radical of $P$ is nilpotent, i.e.,
 $(\bigcap\limits_{i=1}^r (a_i))^N=\textbf{0}$ for some $N\in \mathbb{N}$. This implies that $$\bigcap\limits_{i=1}^r (a_i)^{n_i}=\textbf{0},$$ is the zero ideal of $P$. Combined with \eqref{equation x wide equiv x d 1} and  \eqref{equation x wide equiv x d 2}, we have that $x* d \equiv x* \widetilde{d} \pmod {\textbf{0}}$, and thus $$x * d=x *\widetilde{d}.$$ By \eqref{equation wide d in U(R)} and by taking $u=\widetilde{d}$, we have Conclusion (i) proved.

 \medskip

\noindent (ii). \  Since  $b=a_1^{m_1}* \cdots  * a_r^{m_r} * u \in \bigcap\limits_{i=1}^r (a_i)^{m_i}$, we have that
 $$m_i\leq \zeta(a_i:b)\ \mbox{
 for all } \ i=1,2,\ldots,r.$$ To prove the conclusion, we assume to the contrary that $m_i< \zeta(a_i:b)$ for some $i\in [1,r]$, say $$m_1<\zeta(a_1:b).$$
 It follows that $$(b)\subseteq (a_1)^{\zeta(a_1:b)}\subseteq  (a_1)^{m_1+1}.$$
 Since
 $(a_2)^{m_2}*\cdots *(a_r)^{m_r}+_{R}(a_1)=(1_R)$, it follows that $$\begin{array}{llll}
(a_1)^{m_1+1}&=& (b)+_{R}(a_1)^{m_1+1}\\
&=&(a_1^{m_1}* a_2^{m_2}*\cdots * a_r^{m_r})+_{R}(a_1)^{m_1+1} \\
&=&(a_1)^{m_1}*(a_2)^{m_2}\cdots * (a_r)^{m_r}+_{R}(a_1)^{m} * (a_1) \\
&=&(a_1)^{m_1}*[(a_2)^{m_2}*\cdots * (a_r)^{m_r}+_{R}(a_1)] \\
&=& (a_1)^{m_1}*(1_R) \\
&=& (a_1)^{m_1},\\
\end{array}$$
a contradiction with $m_1<n_1$. Therefore, Conclusion (ii) is proved.

\medskip

\noindent (iii). \ Suppose $(b)\subseteq (c)$. Since any ideal containing $c$ must contain $b$, we have that $\zeta(a_i:b)\geq\zeta(a_i:c)$ for all $i\in [1,r]$. Conversely, the sufficiency follows from Conclusion (i).

\medskip

\noindent (iv). \ By (iii), we derive that $(b)=(c)$ holds if and only if $\zeta(a_i:b)=\zeta(a_i:c)$ for all $i\in [1,r]$. Combined with Conclusion (i), we have this conclusion proved immediately.
\end{proof}

\bigskip

Now we are in a position to prove Theorem \ref{Theorem rings}.

\noindent{\sl Proof of Theorem \ref{Theorem rings}.} \
We first show that $\Gamma(H_a)\cong {\rm U}(R_a)$.
%To avoid the confusion, we shall denote by $+_{R}$ the addition operation and denote by ${\cdot}_{R}$ the multiplication operation in the ring $R$.
Now we consider the multiplicative semigroup $\mathcal{S}_R$ of the ring $R$. Recall that $$\rho_a: {\rm St}(H_a)\rightarrow \Gamma(H_a)$$ is an epimorphism given by $$c\mapsto \gamma_c$$ for any $c\in {\rm St}(H_a)$. Let $$\varphi:R\rightarrow R_a$$ be the canonical epimorphism given by $$\varphi: c\mapsto \overline{c}=c+_{R}{\rm Ann}(a)\in R_a.$$ We see that for any element $c\in R$,
$$\begin{array}{llll}
c\in {\rm St}(H_a)&\Leftrightarrow& (c * a) \ \mathcal{H} \ a \\
&\Leftrightarrow& d * (c * a)=a \ \mbox{ for some } d\in R \\
&\Leftrightarrow& (d * c -_{R} 1_R)* a=0_R \ \mbox{ for some } d\in R \\
&\Leftrightarrow& d * c -_{R} 1_R\in {\rm Ann}(a) \ \mbox{ for some } d\in R\\
&\Leftrightarrow& \overline{d} * \overline{c}=\overline{1_R} \ \mbox{ for some } d\in R \\
&\Leftrightarrow& \overline{c}\in {\rm U}(R_a),\\
\end{array}$$
i.e., the restriction $\varphi |_{{\rm St}(H_a)}$ of $\varphi$ within the domain ${\rm St}(H_a)$ is an epimorphism of ${\rm St}(H_a)$ onto ${\rm U}(R_a)$, for convenience, we still use $$\varphi: {\rm St}(H_a)\rightarrow {\rm U}(R_a)$$ to denote this epimorphism.  By Lemma \ref{lemma semigroup isomorphism theorem}, to set an isomorphism of ${\rm U}(R_a)$ onto $\Gamma(H_a)$, it suffices to show that
 $${\rm ker}\ \varphi={\rm ker} \ \rho_a.$$

We see that for any
$(c,d)\in {\rm St}(H_a)\times {\rm St}(H_a)$,
$$(c,d)\in {\rm ker}\  \varphi\Leftrightarrow \overline{c}=\overline{d}\Leftrightarrow c-_{R} d\in{\rm Ann}(a)\Leftrightarrow (c-_{R} d)*a=0_R \Leftrightarrow c* a=d* a\Leftrightarrow \gamma_c\circ a=\gamma_d\circ a,$$ in addition,  by Conclusion (i) of Lemma \ref{Lemma schutzenburger group}, we see that $\gamma_c\circ a=\gamma_d\circ a$ if and only if, $\gamma_c=\gamma_d$, and equivalently, $(c,d)\in {\rm ker}\ \rho_a$.  This proves that ${\rm ker}\ \varphi={\rm ker}\ \rho_a$ and thus
$${\rm U}(R_a)\cong \Gamma(H_a).$$
Combined with Theorem \ref{Theorem Main theorem abstract},  we have that
if ${\rm U}(R_a)$ is infinite then ${\rm D}_a(\mathcal{S}_R)$ is infinite, and if ${\rm U}(R_a)$ is finite then ${\rm D}_a(\mathcal{S}_R)$ is finite and
\begin{equation}\label{equation Da(R)leq}
{\rm D}_a(\mathcal{S}_R)\leq \Psi(a)+{\rm D}({\rm U}(R_a))-1.
\end{equation}

\medskip

Now we assume that $R$ is a finite commutative principal ideal unitary ring.  Trivially, $R$ satisfies the a.c.c. for ideals and ${\rm U}(R_a)$ is finite.  We need to show that $${\rm D}_a(\mathcal{S}_R)= \Psi(a)+{\rm D}({\rm U}(R_a))-1.$$
By \eqref{equation Da(R)leq}, it suffices to construct an irreducible sequence $T\in \mathcal{F}(\mathcal{S}_R)$ of length $\Psi(a)+{\rm D}({\rm U}(R_a))-1$ with sum $\sigma(T)=a$.

We show the following.

\noindent \textbf{Claim.} \ For any  $\gamma\in \Gamma(H_a)$, there exists some $u\in {\rm U}(R)$ such that $\rho_a(u)=\gamma$.

{\sl Proof of the claim.} \ Let $x=\gamma\circ a$. Since $x\ \mathcal{H} \ a$, i.e., $(x)=(a)$, it follows from Conclusion (iv) of Lemma \ref{Lemma structure of principal} that $x=a * u$ for some $u\in {\rm U}(R)$. Then $$\gamma_u\circ a=\gamma\circ a.$$ By Conclusion (i) of Lemma \ref{Lemma schutzenburger group}, we derive that $\rho_a(u)=\gamma_u=\gamma$. This proves this claim. \qed

Let $(a_1), \ldots,(a_r)$ be the all distinct maximal ideals of $R$. By the above claim, we can take a sequence
$$V\in \mathcal{F}({\rm U}(R))$$
 such that $\rho_a(V)=\prod\limits_{v\mid V} \gamma_v$ is a zero-sum free sequence in the group $\Gamma(H_a)$ with length \begin{equation}\label{equation length of V}
|\rho_a(V)|=|V|={\rm D}(\Gamma(H_a))-1.
\end{equation}
 By Conclusion (i) of Lemma \ref{Lemma schutzenburger group}, there exists an element
\begin{equation}\label{equation b in Ha-a}
b\in H_a\setminus \{a\}
\end{equation}
 such that
 \begin{equation}\label{equation pi(Vb)=a}
 \pi(V\cdot b)=\pi(V)* b=\gamma_{\pi(V)}\circ b=\rho_a(\pi(V))\circ b=a.
 \end{equation}
By \eqref{equation b in Ha-a} and Conclusion (i) of Lemma \ref{Lemma structure of principal}, we derive that
\begin{equation}\label{equation b=a1...aru}
b=a_1^{m_1}* \cdots  * a_r^{m_r} * u,
\end{equation}
 where $$u\in {\rm U}(R)$$ and $$m_i=\zeta(a_i:b)=\zeta(a_i:a) \mbox{ for all } i\in [1,r].$$
By Conclusions (i), (ii) and (iii) of Lemma \ref{Lemma structure of principal}, we conclude that
\begin{equation}\label{equation omega(a)=}
\Psi(a)=\Psi(b)=\Sigma_{i=1}^r \zeta(a_i,b)=\Sigma_{i=1}^r m_i.
\end{equation}
Since $a\notin {\rm U}(R)$, there exists some $i\in [1,r]$ such that $m_i>0$, say $$m_1>0.$$
Let $$T=V\cdot (a_1* u)\cdot a_1^{[m_1-1]}\cdot \prod\limits_{i=2}^r a_i^{[m_i]}.$$
 It follows from \eqref{equation pi(Vb)=a} and \eqref{equation b=a1...aru} that  $$\pi(T)=a,$$ and follows from \eqref{equation length of V} and \eqref{equation omega(a)=} that $$|T|=|V|+\Sigma_{i=1}^r m_i={\rm D}(\Gamma(H_a))-1+\Sigma_{i=1}^r m_i={\rm D}(\Gamma(H_a))-1+\Psi(a).$$

It remains to show that $T$ is an irreducible sequence. Assume to the contrary that $T$ contains a proper subsequence $W$ such that
\begin{equation}\label{equation pi(W)=pi(T)}
\pi(W)=\pi(T).
\end{equation}
 Since $\pi(W)\ \mathcal{H} \ \pi(T)$, it follows from Lemma \ref{Lemma folklore} and Conclusions (ii), (iii), (iv) of Lemma \ref{Lemma structure of principal} that $$TV^{[-1]}\mid W.$$ Then $$W=L\cdot (TV^{[-1]}),$$ where $L$ is a proper subsequence of $V$. Since $\rho_a(V)$ is zero-sum free in the group $\Gamma(H_a)$, it follows that $$\rho_a(\pi(L))\neq \rho_a(\pi(V)).$$
Combined with Conclusion (i) of Lemma \ref{Lemma schutzenburger group},  we derive that  $\pi(W)=\pi(L)*\pi(TV^{[-1]})=\pi(L)* b=\rho_a(\pi(L))\circ b\neq \rho_a(\pi(V))\circ b=\pi(T)$, a contradiction with \eqref{equation pi(W)=pi(T)}. Hence, $T$ is irreducible. This completes the proof of Theorem \ref{Theorem rings}.  \qed

\section{Concluding remarks}

Theorem \ref{Theorem Main theorem abstract} states that for any element $a\in \mathcal{S}$, in the premise of $\Psi(a)$ being finite, ${\rm D}_a(\mathcal{S})$ is finite if and only if $H_a$ is finite. We remark that $\Psi(a)$ being finite is not necessary when ${\rm D}_a(\mathcal{S})$ is finite. For example, let $X=\{x_i:i\in \mathbb{Z}\}$, and let $\mathcal{S}=<X\mid \mathscr{R}>$ be a commutative semigroup generated by $X$ subject to the defining relation $\mathscr{R}$, where $$\mathscr{R}=\{(n+1)x_i=x_i\}\cup \{x_i+x_j=x_j \mbox{ for any } i<j\}.$$ It is not hard to check that for any $a\in \mathcal{S}$, $\Psi(a)$ is infinite but ${\rm D}_a(\mathcal{S})\leq n.$
Hence, a natural question for general commutative semigroups came up.

\medskip

\noindent \textbf{Question 1.} \ Let $\mathcal{S}$ be any commutative semigroup, and let $a$ be an element of $\mathcal{S}$. From the point of view of semigroup's structure, does there exists a sufficient and necessary condition to decide whether ${\rm D}_a(\mathcal{S})$ is finite or infinite?

\medskip

We return now to the constants ${\rm d}(\mathcal{S})$ and ${\rm D}(\mathcal{S})$. As shown in Proposition $D$, in case that $\mathcal{S}$ is a finite commutative semigroup, ${\rm d}(\mathcal{S})$ is finite and  ${\rm D}(\mathcal{S})={\rm d}(\mathcal{S})+1.$ Actually, both constants ${\rm d}(\mathcal{S})$ and ${\rm D}(\mathcal{S})$ relate closely to each other in any commutative semigroups (not necessarily finite), which can be seen from the following.

\medskip

\begin{prop} Let $\mathcal{S}$ be a commutative semigroup. Then ${\rm D}(\mathcal{S})$ is finite if and only if ${\rm d}(\mathcal{S})$ is finite. Moreover, in case that ${\rm D}(\mathcal{S})$ is finite, we have $${\rm D}(\mathcal{S})={\rm d}(\mathcal{S})+1.$$
\end{prop}

\begin{proof} \ Suppose first that ${\rm d}(\mathcal{S})$ is finite. Take an arbitrary sequence $T\in\mathcal{F}(\mathcal{S})$ of length at least ${\rm d}(\mathcal{S})+1$.  By the definition
of ${\rm d}(\mathcal{S})$, there exists a subsequence $T'$ of $T$ with $|T'|\leq {\rm d}(\mathcal{S})<|T|$ such that $\sigma(T')=\sigma(T)$, which implies that $T$ is reducible. By the arbitrariness of $T$, we have ${\rm D}(\mathcal{S})$ is finite and ${\rm D}(\mathcal{S})\leq
{\rm d}(\mathcal{S})+1.$

Now suppose that ${\rm D}(\mathcal{S})$ is finite. Let $V\in \mathcal{F}(\mathcal{S})$ be an arbitrary sequence. Take a shortest subsequence $V'$ of $V$ such that $\sigma(V')=\sigma(V)$. By the minimality of $|V'|$, we have that $V'$ is either the empty sequence $\varepsilon$ or an irreducible sequence. It follows that $|V'|\leq {\rm D}(\mathcal{S})-1$. By the arbitrariness of $V$, we have that ${\rm d}(\mathcal{S})$ is finite and ${\rm d}(\mathcal{S})\leq
{\rm D}(\mathcal{S})-1.$
This competes the proof of this proposition.  \end{proof}

By proposition \ref{proposition Da(S) and D(S)}, we can derive the following proposition immediately.

\begin{prop} \ Let $\mathcal{S}$ be a commutative Noetherian semigroup. Then ${\rm D}(\mathcal{S})$ and ${\rm d}(\mathcal{S})$ is finite if, and only if, $|H_a|$ is bounded for all $a\in \mathcal{S}$, i.e., there exists an integer $\mathcal{M}$ such that $|H_a|<\mathcal{M}$ for all $a\in \mathcal{S}$.
\end{prop}

We close this paper by proposing the following question.

\medskip

\noindent \textbf{Question 2.} \ Let $\mathcal{S}$ be any commutative semigroup. From the point of view of semigroup's structure, does there exists a sufficient and necessary condition to decide whether ${\rm D}(\mathcal{S})$ is finite or infinite?

\bigskip

\noindent {\bf Acknowledgements}

\noindent  This work is supported by NSFC (11301381, 11271207), Science and Technology Development Fund of Tianjin Higher
Institutions (20121003).


\bigskip


\begin{thebibliography}{99}




\bibitem{AdhikariGaoWang14} S.D. Adhikari, W.D. Gao and G.Q. Wang, \emph{Erd\H{o}s-Ginzburg-Ziv theorem for finite commutative semigroups,} Semigroup Forum, \textbf{88} (2014)  555--568.

\bibitem{Cnumber} W.R. Alford, A. Granville and C. Pomerance,  \emph{There are infinitely many Carmichael numbers,} Ann. of Math., \textbf{140} (1994) 703--722.

\bibitem{AlonJctb} N. Alon, S. Friedland and G. Kalai, \emph{Regular subgraphs of almost regular graphs,}
J. Combin. Theory Ser. B, \textbf{37} (1984) 79--91.

\bibitem{Boone} W. Boone, \emph{The word problem,} Ann. of Math., \textbf{70} (1959) 207--265.


\bibitem{Davenport} H. Davenport, Proceedings of the Midwestern conference on group theory and number theory, Ohio State University, April 1966.

\bibitem{Dehn} M. Dehn, \emph{$\ddot{U}$ber unendliche diskontinuierliche Gruppen,} Math. Ann., \textbf{71} (1911) 116--144.

\bibitem{EGZ} P. Erd\H{o}s, A. Ginzburg and A. Ziv, \emph{Theorem in additive number theory,}
Bull. Res. Council Israel 10F (1961) 41--43.


\bibitem{GaoGeroldingersurvey} W.D. Gao and A. Geroldinger, \emph{Zero-sum problems in finite abelian groups: a survey,} Expo. Math.,  \textbf{24} (2006) 337--369.


\bibitem{GH} A. Geroldinger and F. Halter-Koch, \emph{Non-Unique
Factorizations. Algebraic, Combinatorial and Analytic Theory,}
Pure and Applied Mathematics, vol. 278, Chapman $\&$ Hall/CRC,
2006.


\bibitem{GLP12} A. Geroldinger, M. Liebmann and A. Philipp, \emph{On the Davenport constant and on the structure of extremal sequences,}
Period. Math. Hung., \textbf{64} (2012) 213--225.


\bibitem{SemigroupRing} R. Gilmer,
\emph{Commutative semigroup rings,}
Univ. Chicago Press, Chicago, 1984.

\bibitem{Grillet Semigroup} P.A. Grillet, \emph{Semigroups, an introduction to the structure theory,} Dekker, New York, 1995.

\bibitem{Grillet monograph} P.A. Grillet, \emph{Commutative semigroups,} Kluwer Academic Publishers, 2001.


\bibitem{wordAdvance} E.W. Mayr and A.R. Meyer, \emph{The complexity of the word problems for commutative semigroups and polynomial ideals,} Adv. Math., \textbf{46} (1982) 305--329.

\bibitem{Novikov} P.S. Novikov, \emph{On the algorithmic unsolvability of the word problem in group theory,} Trudy
Mat. Inst. Steklov, 44 (1955); English transi, Amer. Math. Soc. Transi.,  \textbf{9} (1958) 1--122.

\bibitem{Post} E.L. Post, \emph{Recursive unsolvability of a problem of Thue}, J. Symbolic Logic, \textbf{12} (1947) 1--11.

\bibitem{rog1}
K. Rogers, \emph{A Combinatorial problem in Abelian groups,} Proc.
Cambridge Phil. Soc., \textbf{59} (1963) 559--562.

\bibitem{IsrealWord} J.C. Rosales and P.A. Garc\'{i}a-S\'{a}nchez, \emph{Presentations for subsemigroups of finitely generated commutative semigroups,} Israel J. Math., \textbf{113} (1999) 269--283.



\bibitem{SavChen}  S. Savchev and F. Chen, \emph{Long zero-sum sequences in finite cyclic groups,}  Discrete Math., \textbf{307} (2007) 2671--2679.

\bibitem{SkalbaGraz} M. Ska{\l}ba, \emph{The relative Davenport's constant of the group $Z_n\times Z_n$,} Grazer Math. Berichte, \textbf{318} (1992) 167--168.

\bibitem{Skalba}  M. Ska{\l}ba, \emph{On numbers with a unique representation by a binary quadratic form,}  Acta Arith., \textbf{64} (1993) 59--68.

\bibitem{SkalbaEuropean}  M. Ska{\l}ba, \emph{On the Relative Davenport Constant,}  European J. Combin., \textbf{19} (1998) 221--225.


\bibitem{wangDavenportII}  G.Q. Wang, \emph{Davenport constant for semigroups II,}  J. Number Theory, \textbf{153} (2015) 124--134.

\bibitem{wang}  G.Q. Wang, \emph{Structure of the largest idempotent-free sequences in finite semigroups,}  arXiv:1405.6278.

\bibitem{wanggao} G.Q. Wang and W.D. Gao,
\emph{Davenport constant for semigroups,} Semigroup Forum,
\textbf{76} (2008) 234--238.

\bibitem{Yuanpingzhi} P.Z. Yuan, \emph{On the index of minimal zero-sum sequences over finite cyclic groups,} J. Combin. Theory Ser. A, \textbf{114} (2007) 1545--1551.




\end{thebibliography}
\end{document}